\documentclass{amsart}
\usepackage[utf8]{inputenc}
\usepackage[utf8]{inputenc}
\usepackage{amssymb}
\usepackage{amsmath}
\usepackage{amsthm, amsfonts, mathrsfs}
\usepackage{amsmath}
\usepackage{amssymb}
\usepackage{amsthm}
\usepackage{graphicx}\usepackage[colorinlistoftodos]{todonotes}

\usepackage{amsrefs}
\newtheorem{thm}{Theorem}[section]
\newtheorem{cor}[thm]{Corollary}

\newtheorem{prop}[thm]{Proposition}

\newtheorem{qu}[thm]{Question}
\newtheorem{ex}[thm]{Example}

\usepackage{CJKutf8}

\begin{document}
\title{Sub-posets in $\omega^\omega$ and The Strong Pytkeev$^\ast$ Property}

\author[Z. Feng]{Ziqin Feng}
\address{Department of Mathematics and Statistics\\Auburn University\\Auburn, AL~36849}
\email{zzf0006@auburn.edu}

\author[P. Nukala]{Naga Chandra Padmini Nukala}
\address{Department of Mathematics and Statistics\\Auburn University\\Auburn, AL~36849}
\email{nzn0025@auburn.edu}

\begin{abstract}Tukey order are used to compare the cofinal complexity of partially order sets (posets). We prove that there is a $2^\mathfrak{c}$-sized collection of sub-posets in $2^\omega$ which forms an antichain in the sense of Tukey ordering. Using the fact that any boundedly-complete sub-poset of $\omega^\omega$ is a Tukey quotient of $\omega^\omega$, we answer two open questions published in \cite{FKL16}.

The relation between $P$-base and strong Pytkeev$^\ast$ property is investigated. Let $P$ be a poset  equipped  with  a second-countable topology in which every  convergent  sequence   is  bounded. Then we prove that any topological space with a $P$-base has the strong Pytkeev$^\ast$ property. Furthermore, we prove that  every  uncountably-dimensional locally convex space (lcs) with a $P$-base contains an infinite-dimensional metrizable compact subspace. Examples in function spaces are given.

\end{abstract}
\date{\today}


\begin{CJK*}{UTF8}{gbsn}

\end{CJK*}
\maketitle
\section{Introduction}

The neighborhood base at a point $x$ in a topological space $X$ is a poset ordered by inverse set-inclusion. Posets, for example $\omega$, $\omega^\omega$, and $\mathcal{K}(M)$ with $M$ being a  separable metric space, are used to measure the cofinal complexity of the neighborhood base at a point. Here, $\omega^\omega$ is the family of all sequences of natural numbers in the pointwise partial order, i.e. $f\leq g$ if and only if $f(n)\leq g(n)$ for all $n\in\omega$; $\mathcal{K}(M)$ is the collection of compact subsets of $M$ ordered by set-inclusion. Tukey order (see the definition in Section~\ref{prep}) which was introduced  in \cite{Tuk40}  early 20th century is a powerful tool  to compare the cofinal complexity of posets. Using Tukey order, a space is first-countable at a point $x$ if the neighborhood base poset at $x$ is a Tukey quotient of $\omega$. All the topological spaces in this paper are assumed to be Tychonoff.

Let $P$ be an arbitrary poset. In general, a topological space $X$ is defined to have a neighborhood $P$-base at an $x\in X$ if the neighborhood base poset at $x$ is a Tukey quotient of $P$; and, the space $X$ has a $P$-base if it has a neighborhood $P$-base at each $x\in X$. Clearly, a topological group has a $P$-base if it has a neighborhood $P$-base at the identity. The concept of $\omega^\omega$-base (or, $\mathfrak{G}$-base) was introduced in \cite{FKLS06} for locally convex spaces. Topological spaces with an $\omega^\omega$-base have been intensively studied in recent years (see \cite{Banakh2019}, \cite{BL18}, \cite{LPT17}, and \cite{GKL15}). The authors in \cite{DF20} investigated the compact spaces with a  $P$-base for some other posets, mainly $\mathcal{K}(M)$ where $M$ is a separable metric space. In this paper, we consider the sub-posets in $\omega^\omega$ and the relation between $P$-base and the strong Pytkeev$^\ast$ networks (see definition in Section~\ref{sPp}). After the characterization of quasi-barrreled lcs with an $\omega^\omega$-base was given in \cite{CKS03}, lcs with an $\omega^\omega$-base attracted lots of attentions in functional analysis. In the end of this paper, we investigate the property of compact subsets in lcs with $P$-bases for some general poset $P$.

This paper is organized in the following way. In Section~\ref{subp}, we discuss the sub-posets in $2^\omega$ and prove that there is a $2^\mathfrak{c}$-sized collection of sub-posets in $2^\omega$ which forms an antichain in the sense of Tukey ordering. Then, we answer two open questions posted in \cite{FKL16} by showing that: 1) any topological group with a $\Sigma_2$-base admits an $\omega^\omega$-base; and 2), any separable metric space $M$ is polish if $\mathcal{K}(M)\leq_T \Sigma$ for any unbounded and boundedly-complete proper sub-poset $\Sigma$ in $\omega^\omega$.

   In Section~\ref{sPp}, we discuss the strong Pytkeev$^\ast$ property (see the definition in Section~\ref{sPp}). We give a sufficient condition for a free filter on a countable set to be meager and then use it to prove that  any space $X$ with a neighborhood $P$-base at $x$ has a strong Pytkeev$^\ast$ network at the point given that  $P$ is a directed set equipped with a second-countable topology in which every  convergent  sequence  in $P$ is  bounded. Hence any topological space with such a $P$-base has the strong Pytkeev$^\ast$ property. We give an example of a topological space with a $\mathcal{K}(\mathbb{Q})$-base but doesn't have the the strong Pytkeev property, here $\mathbb{Q}$ is the space of rationals.  Then we show that each uncountably-dimensional locally convex space (lcs) $E$ with a $P$-base contains an infinite-dimensional metrizable compact subspace if  $P$ is a directed set equipped with a second-countable topology in which  every  convergent  sequence  in $P$ is  bounded. This extended Theorem 1.2 in \cite{BKP20}. Examples in function spaces are given.

\section{Preliminaries}\label{prep}

For any topological space $M$, the poset $\mathcal{K}(M)$ is the collection compact subsets of $X$ ordered by set-inclusion $\subseteq$. Hence,  a topological space $X$ has a neighborhood $\mathcal{K}(M)$-base at a point $x\in X$, if there exists a neighborhood base $\{\mathcal{U}_{K}(x):{K\in\mathcal{K}(M)}\}$ at $x$ such that $\mathcal{U}_{L}(x)\subseteq\mathcal{U}_{K}(x)$ for all $K\subseteq L$ in $\mathcal{K}(M)$.

For any topological space $X$, $C(X)$ is the collection of all continuous real-valued functions. We use $C_k(X)$ ($C_p(X)$) to denote the topological space $C(X)$ endowed with the compact-open topology (resp., pointwise convergence topology). A basic open neighborhood of $f\in C_k(X)$ ($C_p(X)$) is the form $B(f, K, \epsilon)=\{g: |f(x)-g(x)|<\epsilon \text{ for  all }x\in K\}$ where $\epsilon>0$ and $K$ is a compact (resp., finite) subset in $X$.

A filter on a set $X$ is a collection $\mathcal{F}$ of non-empty subsets of $X$ which is closed under finite intersections and taking supersets in $X$. A free filter is a filter with empty intersection. Identifying each subset of $X$ with its characteristic function in $2^X$, we can equip $\mathcal{P}(X)$ with a compact Hausdorff topology. If $X$ is countable, we could identify $\mathcal{P}(X)$ with $2^\omega$ which is a separable metric space.

Let $P$ be a directed poset, i.e. for any points $p, p'\in P$, there exists a point $q\in P$ such that $p\leq q$ and $p'\leq q$.  A subset $C$ of $P$ is \emph{cofinal} in $P$ if for any $p\in P$, there exists a $q\in C$ such that $p\leq q$. Then $\text{cof}(P)=\min\{|C|: C\text{ is cofinal in }P\}$. We also define $\text{add}(P)=\min\{|Q|: Q \text{ is unbounded in }P\}$. Given two directed sets $P$ and $Q$, we say $Q$ is a Tukey quotient of $P$, denoted by $P\geq_T Q$, if there is a map $\phi:P\rightarrow Q$, a Tukey quotient,  which carries cofinal subsets of $P$ to cofinal subsets of $Q$. A map $\psi: Q\rightarrow P$ is said to be a Tukey map if $\psi(U)$ is unbounded in $P$ whenever $U$ is unbounded in $Q$. It is known that $Q$ is a Tukey quotient of $P$ if and only if there is a Tukey map from $Q$ to $P$. If $P$ and $Q$ are mutually Tukey quotients, we say that $P$ and $Q$ are Tukey equivalent, denoted by $P=_T Q$. Gartside and Mamatelashvili in \cite{GM16} constructed a $2^\mathfrak{c}$-sized antichain in $\mathcal{K}(\mathcal{M})=\{\mathcal{K}(M): M \text{ is a separable and metrizable space}\}$.

\begin{prop}\label{cof_tu_e} Let $Q$ be a sub-poset of a poset $P$. If $Q$ is cofinal in $P$, then $P=_TQ$.
\end{prop}

\begin{proof} Let $Q$ be a cofinal subset of $P$. It is clear that the inclusion mapping from $Q$ to $P$ is both a Tukey quotient and a Tukey map. Hence $Q$ and $P$ are Tukey equivalent. \end{proof}


The posets in our discussion are naturally equipped with a topology, for example, $\omega^\omega$ with product topology and $\mathcal{K}(M)$ with Vietoris topology. One property of topological posets which plays important role in several proofs is that every convergent sequence is bounded.  This property is weaker than being  $\aleph_0$-directed, i.e., every countable set is bounded. It is straightforward to see that $\omega^\omega$ with product topology satisfies this property. It is proved \cite{Banakh2019} that for any monotone function $\phi:\omega^\omega\rightarrow P$ with cof$(P)\leq \omega$ and $f\in\omega^\omega$ there is an open neighborhood $O$ of $f$ such that $\phi(O)$ is bounded. Next we prove this result holds for any poset with a first-countable topology in which every convergent sequence is bounded.
\begin{prop}\label{bdedness} Let $P$ be a directed set equipped with a first-countable topology such that every convergent sequence is bounded. Suppose that $Q$ is a directed set with $\text{add}(Q)=\omega$ and $f$ is any monotone mapping from $P$ to $Q$. For any $p\in P$, there exists a neighborhood $B$ of $p$ such that $f(B)=\{f(b): b\in B\}$ is bounded in $Q$.
\end{prop}

\begin{proof} Let $\{q_n: n\in\omega\}$ be an unbounded subset in the poset $Q$. Take a monotone mapping $f$ from $P$ to $Q$. Fix $p\in P$ and a countable decreasing local basis $\{B_n : n\in\omega\}$ at $p$.

Suppose, for a contradiction, that $f(B_n)$ is unbounded for each $n\in \omega$. For each $n\in\omega$, pick $b_n\in B_n$ such that $f(b_n)\geq q_n$. Then $\{b_n: n\in \omega \}$ is a sequence in $P$ which converges to $p$. By the assumption, there is a $p^\ast\in P$ such that $p^\ast\geq b_n$ for each $n$ and $p^\ast\geq p$. Since $f$ is monotone, $f(p^\ast)\geq q_n$ for each $n\in\omega$. This is a contradiction because   $\{q_n: n\in\omega\}$ is unbounded in the directed set $Q$. \end{proof}

If $M$ is a separable metric space, the poset $\mathcal{K}(M)$ is also a separable metric space with the Hausdorff metric. It is straightforward to see that if $\{K_n: n\in\omega\}$ is a sequence in $\mathcal{K}(M)$ converging to $K$ then $K\cup \bigcup\{K_n: n\in \omega\}$ is also compact. Hence every convergent sequence in $\mathcal{K}(M)$ is bounded above. By the lemma above, we get the following corollary.

\begin{cor}\label{M_bd} Let $M$ be a separable metric space. 
For every monotone function $f:\mathcal{K}(M)\rightarrow Q$ where cof$(P)\leq \omega$ and every $K\in\mathcal{K}(M)$, there exists a neighborhood $B$ of $K$ such that $f(B)$ is bounded.
\end{cor}

A family $\mathcal{N}$ of subsets of a topological space $X$ is said to be a Pytkeev-network at $x$ if for any neighborhood $O_x$ of $x$ and each subset $A\subset X$ with $x \in \overline{A}$ there is a set $N\in \mathcal{N}$ such that $x \in  N \subset O_x$ , $N \cap A \neq\emptyset$, and moreover $N\cap A$ is infinite if the set $A$ accumulates at $x$.  A space $X$ is said to have the strong Pytkeev property at $x\in X$ if it has a countable Pytkeev-network at $x$; and, we say $X$ has the strong Pytkeev property if it has the strong Pytkeev property at any $x\in X$.


\section{Sub-posets in  $2^\omega$ and $\omega^\omega$}\label{subp}

In this section, we consider the Tukey class of sub-posets in $2^\omega$, more specifically its relation with the Tukey class of $\mathcal{K}(M)$ where $M$ is a separable metric space.

\begin{prop}\label{subo-sepmet} For any separable metric space $M$, there is a sub-poset $\Sigma$ in $2^\omega$ such that $\Sigma=_T \mathcal{K}(M)$.  \end{prop}

\begin{proof} Let $M$ be a separable metric space and  $D=\{d_n: n\in \omega\}$ be an enumeration of a countable dense subspace of $M$. Define $\Sigma=\{Z\subset \omega: \overline{\{d_n: n\in Z\}} \text{ is compact in }M \}$ and a map $\phi: \Sigma\rightarrow \mathcal{K}(M)$ such that $\phi(Z)=\overline{\{d_n: n\in Z\}}$. We claim that $\phi$ is both a Tukey quotient and a Tukey map from $\Sigma$ to $\phi(\Sigma)$, i.e., $\Sigma=_T\phi(\Sigma)$.

Let $\mathcal{C}$ be a cofinal subset of  $\Sigma$. Pick a $K\in \phi(\Sigma)$. Then take any $Z \in \Sigma$ such that $\phi(Z)=K$.  Then there is a $C \in \mathcal{C}$ such that $C\supseteq Z$. Hence $K\subset \phi(C)$. Therefore, $\phi(\mathcal{C})$ is cofinal in $\phi(\Sigma)$. Hence $\phi$ is a Tukey quotient.

Let $\mathcal{U}$ be an unbounded subset of $\Sigma$. Assume, for a contradiction, that $\phi(\mathcal{U})$ is bounded in $\phi(\Sigma)$, i.e., there is a compact subset $K$ of $M$ such that $K\in\phi(\Sigma)$ and  $\phi(Z)\subset K$ for each $Z\in\mathcal{U}$. Define $Z_0=\bigcup \{Z\in \Sigma: \phi(Z)\subset K\}$. Then $\overline{\{d_n: n\in Z_0\}}\subseteq K$, hence, it is compact. Therefore,    $Z_0\in \Sigma$ and $Z_0\supset Z$ for any $Z\in \mathcal{U}$ which contradicts with the unboundedness of $\mathcal{U}$.  Therefore $\phi$ is a Tukey map.

\medskip Then we prove that $\phi(\Sigma)$ is cofinal in $\mathcal{K}(M)$, hence by Proposition~\ref{cof_tu_e}, $\phi(\Sigma)=_T \mathcal{K}(M)$, furthermore, $\Sigma=_T\mathcal{K}(M)$. Fix $K\in \mathcal{K}(M)$. Let $\{k_i: i\in\omega\}$ be a countable dense subset of $K$.  For each $i\in\omega$, we fix a sequence $\{a_{i, j}: j\in\omega\}$ in $D$ such that the sequence converges to $k_i$ for each $i\in\omega$ and the distance between $a_{i, j}$ and $k_i$ is $\leq 1/j$ . Let $S=\{a_{i, j}: i,j\in\omega\text{ with } j\geq i\}$. Then we define $Z_K=\{n: a_{i, j}=d_n \text{ for all }i,j\in\omega\text{ with } j\geq i\}$ and $H=\overline{\{d_n: n\in Z_K\}}$. Clearly $K\subset H$.  We claim that $H$ is compact. Let $\mathcal{O}$ be an open cover for $H$. Then, there exists a finite subcollection $\mathcal{O}'$ of $\mathcal{O}$ such that $K\subseteq \bigcup \mathcal{O}'$. Then it is straightforward to verify that $H\setminus \bigcup\mathcal{O}'$ is finite, hence $H$ is compact. Therefore, $Z_K$ is in $\Sigma$ and this shows that $\phi(\Sigma)$ is cofinal in $\mathcal{K}(M)$.  \end{proof}


\begin{thm} There is a $2^{\mathfrak{c}}$-sized family, $\mathcal{A}$, of sub-posets in  $2^\omega$ such that for distinct $P$ and $Q$ from $\mathcal{A}$ we have $P\ngeq_T Q$ and $P\nleq_T Q$.
\end{thm}

\begin{proof} Let $\mathcal{M}$ be a $2^\mathfrak{c}$-sized collection of separable metric spaces such that for any pair $M, N\in \mathcal{M}$, $\mathcal{K}(M)\neq_T\mathcal{K}(N)$. By Lemma~\ref{subo-sepmet}, we pick a subset $\Sigma_M$ of $2^\omega$ for each $M\in\mathcal{M}$ such that $\Sigma_M=_T\mathcal{K}(M)$. Then $\{\Sigma_M: M\in\mathcal{M}\}$ is a $2^\mathfrak{c}$-sized antichain in the collection $\mathcal{P}(2^{\omega})$ in term of Tukey ordering.
\end{proof}

A subset $\Sigma$ of $\omega^{\omega}$ is boundedly complete if each subset of $\Sigma$  has a bound in $\Sigma$ given that it is bounded in $\omega^\omega$.  A $\Sigma$-base is called a $\Sigma_2$-base if $\Sigma$ is a boundedly-complete sub-poset in $\omega^\omega$. The concept of $\Sigma_2$-bases was introduced in \cite{FKL16}.  Next result shows that any boundedly-complete sub-poset of $\omega^\omega$ is a Tukey quotient of $\omega^\omega$.
\begin{thm}\label{sigma2} Any boundedly complete sub-poset of $\omega^\omega$ is a Tukey quotient of $\omega^\omega$.
\end{thm}

\begin{proof} Let $\Sigma$ be a boundedly complete sub-poset of $\omega^\omega$. Consider the inclusion map $i:\Sigma\rightarrow \omega^\omega$. Because $\Sigma$ is boundedly complete, a subset in $\Sigma$ is unbounded if and only if it is unbounded in $\omega^\omega$. Hence the inclusion map is a Tukey map. Hence $\Sigma$ is a Tukey quotient of $\omega^\omega$, i.e., $\Sigma\leq_T \omega^\omega$.
\end{proof}

Let $P$ and $Q$ be posets such that $P\leq_T Q$. It is straightforward to see that any space with a $P$-base also has a $Q$-base (see Proposition 2.1 in \cite{DF20}). Hence, we get the following result.

\begin{cor} Every topological space with a $\Sigma_2$-base admits an $\omega^\omega$-base, i.e., a $\mathfrak{G}$-base. \end{cor}

As a special case, any topological group with a $\Sigma_2$-base admits an $\omega^\omega$-base. This gives a positive answer to Problem 23 in \cite{FKL16}.

In his book \cite{Ch74}, Christensen proved (without using the Tukey order notation) that: for any separable metric space $M$, $\omega^\omega\geq_T \mathcal{K}(M)$ if and only if $M$ is Polish (in
other words, completely metrizable). By this result and Theorem~\ref{sigma2}, we  prove the following result which answers Problem 24 in \cite{FKL16}.

\begin{thm} Let $M$ be a separable metric space admitting a compact ordered covering of $M$ indexed by an unbounded and boundedly-complete proper subset of $\omega^\omega$ that swallows the compact sets of $M$. Then the space $M$ is  a Polish space.
\end{thm}

\begin{proof} Let $\Sigma$ be any unbounded and boundedly-complete proper subset of $\omega^\omega$. Let $\{K_f: f\in \Sigma\}$  be a $\Sigma$-ordered compact covering  which swallows the compact sets of  $M$, i.e., $\{K_f: f\in \Sigma\}$ is cofinal in $\mathcal{K}(M)$. By Proposition~\ref{cof_tu_e}, $\{K_f: f\in \Sigma\}=_{T}\mathcal{K}(M)$. The underlying mapping from $\Sigma$ to $\{K_f: f\in \Sigma\}$ is clearly a Tukey quotient, hence $\Sigma\geq_T \{K_f: f\in \Sigma\}$, furthermore,  $\Sigma\geq_T\mathcal{K}(M)$. Therefore,  $\omega^\omega\geq_T\mathcal{K}(M)$ by Theorem~\ref{sigma2}. Therefore, the space $M$ is Polish by Christensen's theorem.  \end{proof}

\section{The Strong Pytkeev$^\ast$ Property}\label{sPp}

In this section, we investigate the relation between $P$-bases and the strong Pytkeev$^\ast$ property. The strong Pytkeev$^\ast$ property is introduced in \cite{Banakh2016}, and also in \cite{Banakh2019}; and it is implied by the strong Pytkeev property. It is proved in \cite{Banakh2016} that any topological space $X$ has the strong Pytkeev property if and only if it is countably tight and has the strong Pytkeev$^\ast$ property. A family $\mathcal{N}$ of subsets of $X$ is said to be a Pytkeev$^\ast$-network (or an s$^\ast$)-network) at $x$ if for any neighborhood $O_x$ of $x$ and any sequence $\{x_n: n\in \omega\}$ that accumulates at $x$ there exists $N\in \mathcal{N}$ such that $N\subset O_x$ and the set $\{n: x_n\in N\}$ is infinite.  A space $X$ is said to have the strong Pytkeev$^\ast$ property at $x\in X$ if it has a countable Pytkeev$^\ast$-network at $x$; and, we say $X$ has the strong Pytkeev$^\ast$ property if it has the strong Pytkeev$^\ast$ property at any $x\in X$.

First we give a sufficient condition for a free filter on a countable set being meager. It is known (see \cite{Banakh2019}) that for a submetrizable space $X$, $X$ has a compact resolution if and only if it is analytic. Also each analytic free filter on a countable set is meager. Hence, every free filter on a countable set with a compact resolution is meager. A topological space has a compact resolution if it has a $\omega^\omega$-ordered compact covering, i.e., a family $\{K_f: f\in \omega^\omega\}$ of compact subsets of  $X$ such that $X=\bigcup\{K_f: f\in\omega^\omega\}$ and $K_f\subseteq K_g$ for every $f\leq g$ in $\omega^\omega$. With an extra condition, this result could be extended to directed set equipped with a second-countable topology in which  every convergent sequence  is bounded.

For the space $2^{X}$, a basic open neighborhood could be represented as $[I, A]=\{S:I\subset S\subset X\text{ and }S\cap A=\emptyset \}$ where $I$ and $A$ are finite subsets of $X$.

\begin{prop}\label{meager} Let $P$ be a directed set equipped with a second-countable topology in which  every convergent sequence is bounded. A free filter $\mathcal{F}$ on a countable set is meager if it has a $P$-ordered compact covering $\mathcal{K}=\{K_p: p\in P\}$ with $\bigcap K_p$ being infinite for each $p\in P$.  \end{prop}

\begin{proof} Let $\mathcal{F}$ be a free filter on $\omega$ and $\mathcal{K}=\{K_p: p\in P\}$ be a $P$-ordered compact covering of $\mathcal{F}$. We first show that for each $p\in P$, there is a neighborhood $O$ of $p$ such that $\bigcup \{K_q: q\in O\}$ is nowhere dense in $2^\omega$. Fix $p\in P$. Assume, for a contradiction, that $\overline{\bigcup \{K_q: q\in O\}}$ has nonempty interior for each neighborhood $O$ of $p$. We pick a decreasing local base $\{O_n: n\in \omega\}$ at $p$ and define $C_n=\bigcup \{K_q: q\in O_n\}$ for each $n\in\omega$. Then, for each $n\in \omega$ there exist finite subsets $I_n$ and $A_n$ of $\omega$ such that $[I_n, A_n]\subset \overline{C_n}$. Without loss of generality, we could assume that $I_n\neq \emptyset$ and $\max I_{n+1}\geq \max I_n+n$ for each $n\in \omega$. Fix  $n\in \omega$. The finite set $I_n$ is an element in $\overline{C_n}$.  Then $U_n=[I_n,\{m: \max I_n <m\leq \max I_{n+1}\}]$ is a basic neighbor of $I_n$ in $2^{\omega}$. Hence, there exists a $p_n\in O_n$ and an element $G_n\in K_{p_n}$ such that $G_n\in U_n $. Note that $m\notin G_n$ for any $m$ with $\max I_n<m\leq \max I_{n+1}$. Hence if $m>\max I_0$, $m\notin \bigcap \{G_n: n\in\omega\}$, i.e., $\bigcap \{G_n: n\in\omega\}$ is finite. Since $\{O_n: n\in \omega\}$ is a decreasing local base at $p$, the sequence $\{p_n: n\in\omega\}$ converges to $p$. By the assumption, there is a $p^\ast\in P$ such that $p^\ast\geq p_n$ for each $n\in\omega$. Therefore, $G_n\in K_{p^\ast}$ for each $n\in\omega$. Then $\bigcap K_{p^\ast}\subset \bigcap \{G_n: n\in \omega\}$ is finite which is a contradiction.

Let $\mathcal{B}=\{B_n: n\in\omega\}$ be a countable base for $P$ and $\mathcal{H}$ be the collection of $\bigcup\{K_q: q\in B_n\}$ which is nowhere dense. By the result above, $\mathcal{F}=\bigcup \mathcal{H}$, hence $\mathcal{F}$ is a meager set. \end{proof}

We would need the famous characterization of meager filters due to Talagrand in \cite{Tal80}. A filter $\mathcal{F}$ on a countable set $X$ is meager subset of the powerset $\mathcal{P}(X)$ if and only if there exists a finite-to-one map $\phi: X\rightarrow \omega$ such that $\phi(F)$ has a finite complement in $\omega$ for any $F\in \mathcal{F}$.  A function from $X$ to $Y$ is said to be finite-to-one if $f^{-1}(y)$ is finite for each $y\in Y$.

\begin{thm}\label{P-Pn} Let $P$ be a directed set equipped with a second-countable topology in which  every convergent sequence is bounded. If the topological space $X$ has a $P$-base at $x$, then it has a countable Pytkeev$^\ast$-network at $x$.
\end{thm}
\begin{proof}  Fix a countable basis $\mathcal{B}=\{B_m:m\in\omega\}$ for $P$. Let $\{U_p:p\in P\}$ be a $P$-ordered local base at $x$. Define $U_{B_m}=\bigcap \{U_p: p\in B_m\}$. We show that $\{U_{B_m}: m\in\omega\}$ is a s$^\ast$-network at $x$. Fix a sequence $\{x_n:n\in \omega\}$ which accumulates at $x$. It is sufficient to show that for every $p\in P$, there is an $m\in \omega$ such that $p\in B_m$ and $U_{B_m}$ contains infinitely many terms in the sequence $\{x_n: n\in\omega\}$.

 If $\{n : x_n=x\}$ is infinite, then any $B_m$ with $p\in B_m$ satisfies that $U_{B_m}$ contains infinite many terms in the sequence $\{x_n: n\in\omega\}$. Without loss of generality, we assume that $x_n\neq x$ for each $n\in\omega$. For each $p\in P$, define $F(U_p)=\{n : x_n\in U_p\}$; and then, let $\mathcal{F}$ be the filter generated by $\{F(U_p): p\in P\}$. Clearly $\mathcal{F}$ is a free filter. For each $p\in P$, define $K_p=\{F: F\in\mathcal{F} \text{ and } F\supset F(U_p)\}$. Then $K_p$ is compact and $\bigcap K_p=F(U_p)$ is infinite for each $p\in P$; also, $K_p\subseteq K_{p'}$ if $p\leq p'$.  So $\mathcal{F}$ is a free filter with a $P$-ordered compact covering which satisfies the requirements in  Lemma~\ref{meager}, hence  it is a meager set.

 \medskip
  By Talagrand's characterization of meager filters, there is a finite-to-one mapping $\phi: \omega\rightarrow \omega$ such that $\phi(F)$ has finite complement in $\omega$ for each $F\in \mathcal{F}$.  Fix $p\in P$. Lastly,  we show that there is a $B_m\in \mathcal{B}$ such that $p\in B_m$ and $U_{B_m}$ contains infinitely  terms in the sequence $\{x_n: n\in \omega\}$. Let $\{O_{k}: k\in \omega\}$ be a subcollection of $\mathcal{B}$ which is a decreasing local base at $p$.  Assume, for a contradiction, that $U_{O_k}$ is finite for each $k\in\omega$. Let $J_k=\{n: x_n\in U_{O_k}\}$. Because $\phi$ is finite-to-one, there is a $y_k\in \omega$ such that $\phi^{-1}(y_k)\neq\emptyset$ and  $J_k\cap \phi^{-1}(y_k)$ is finite for each $k\in \omega$; moreover, we could assume that $y_k<y_{k+1}$ for each $k\in \omega$. Because $P$ is a directed set and $\phi^{-1}(y_k)$ is finite for each $k\in\omega$, there exists a $p_k\in O_k$ such that $x_{i}\notin U_{p_k}$ for each $i\in \phi^{-1}(y_k)$, i.e., $F(U_{p_k})\cap\phi^{-1}(y_k)=\emptyset$ . Clearly, $\{p_k: k\in \omega\}$ is a sequence in $P$ which converges to $p$. Pick $p^\ast \in P$ such that $p^\ast\geq p$ and $p^\ast\geq p_k$ for each $k\in \omega$, hence $F(U_{p^\ast})$ is a subset of $F(U_{p_k})$ for each $k\in\omega$. Therefore, $\phi^{-1}(y_k)\cap F(U_{p^\ast})=\emptyset$ for each $k\in\omega$, i.e., $\{y_k: k\in\omega\}$ is in the complement of $\phi(F(U_{p^\ast}))$. This contradicts the finite-to-one property of the function $\phi$.  \end{proof}

\begin{cor}\label{str_past} Let $P$ be a directed set equipped with a second-countable topology in which  every convergent sequence in $P$ is bounded. If the topological space $X$ has a $P$-base, then it has strong Pytkeev$^\ast$ property.  \end{cor}

Applying this result to topological groups, we obtain the following corollary.
\begin{cor}\label{p_b_sPp} Let $P$ be a directed set equipped with a second-countable topology in which  every convergent sequence in $P$ is bounded.  If topological group $G$ with a $P$-base at the identity, then $G$ has the strong Pytkeev$^\ast$ property.
\end{cor}

It is proved in \cite{sak08} that $C_p(X)$ has the strong Pytkeev property if and only if $X$ is countable, i.e., $C_p(X)$ is metrizable. Hence, it is reasonable to investigate a function space with compact-open topology. It is straightforward to verify that $C_k(M)$ has a $\mathcal{K}(M)$-base for any second-countable space $M$, hence the strong Pytkeev$^\ast$ property by the corollary above. It is proved in   \cite{F20} that if $M$ is second-countable, $C_k(M)$ is countably tight, hence has the strong Pytkeev property. So we greatly generalized Theorem 1 in \cite{TZ09} which shows that $C_k(\omega^\omega)$ has the strong Pytkeev property.

\begin{thm} Let $M$ be any separable metric space. Then $C_k(M)$ has the strong Pytkeev property. \end{thm}

Next, we give an example of $C_k(X)$ which has a $\mathcal{K}(M)$-base for some separable metric space $M$, hence the strong Pytkeev$^\ast$ property, but it doesn't have the strong Pytkeev property.
\begin{ex} There is a topological group with a $\mathcal{K}(M)$-base for some separable metric space $M$, but not the strong Pytkeev property. \end{ex}

\begin{proof} Consider the function space $C_k(\omega_1)$.  Using the fact $\omega_1\leq_T\mathcal{K}(\mathbb{Q})$ (see \cite{GM17}), it is straightforward to verify that  $C_k(\omega_1)$ has a $\omega_1\times \omega$-base, hence a  $\mathcal{K}(\mathbb{Q})\times \omega$-base where $\mathbb{Q}$ is the space of rationals. Therefore, $C_k(\omega_1)$ has a $\mathcal{K}(\mathbb{Q})$-base since $\mathcal{K}(\mathbb{Q})=_T\mathcal{K}(\mathbb{Q})\times \omega$. Therefore, $C_k(\omega_1)$ has the strong Pytkeev$^\ast$ property by the corollary above. It is known (see \cite{BL17}) that the strong Pytkeev property implies countable tightness. Because  $C_k(\omega_1)$ is not countably tight (see Example 9 in \cite{F20}), it doesn't have the strong Pytkeev property. \end{proof}

It has been proven that any topological group with the strong Pytkeev property admits a quasi-$\mathfrak{G}$-base $\{U_f: f\in \Sigma\}$ of the identity, i.e., a neighborhood $\Sigma$-base of the identity for some sub-poset $\Sigma\subset \omega^\omega$. Motivated by Theorem~\ref{P-Pn}, it is natural to ask the following question.

\begin{qu} Let $G$ be a topological group with the strong Pytkeev property. Is it true that $G$ has a $\Sigma$-base where $\Sigma$ is some sub-poset in $\omega^\omega$ such that every convergent sequence is bounded? \end{qu}




Cascales and Orihuela in \cite{CO87} proved that the precompact subsets of any lcs with an $\omega^\omega$-base is metrizable. This result is generalized in \cite{Feng20} which proves that any precompact set of a topological group with a $P$-base is metrizable if $P$ satisfies calibre~$(\omega_1, \omega)$. A directed set $P$ satisfies calibre~$(\omega_1, \omega)$ if every uncountable set in $P$  contains a countable bounded subset. It is straightforward to verify the following result.

\begin{prop}\label{calw} Let $P$  be a directed set equipped with a second-countable topology in which every convergent sequence is bounded. Then $P$ has calibre~($\omega_1, \omega$).
\end{prop}
\begin{proof} Let $S$ be an uncountable subset of $P$. Since $P$ is second-countable, $S$ has a cluster point in $P$, namely, $p_\ast$. Then there is a sequence $\{p_n: n\in \omega\}\subset S$ which converges to $p_\ast$ which is bounded by the assumption. Hence $P$  has calibre~($\omega_1, \omega$). \end{proof}

It is known (see corollary 15.5 in \cite{KKL11}) that a barrelled lcs with an $\omega^\omega$-base is metrizable if and only if it doesn't contains a copy of $\varphi$ which is an $\aleph_0$-dimensional vector space endowed with the finest locally convex topology. Also, any compact subset of $\varphi$ is finite-dimensional. Motivated by these results, it is natural to investigate the class of lcs in which every infinite-dimensional subspace contains an infinite-dimensional compact metrizable subset. In \cite{BKP20}, one of the main results is that every uncountably-dimensional lcs with an $\omega^\omega$-base contains an infinite-dimensional compact metrizable subset. Hence $C_k(X)$ satisfies this property if $X$ is an infinite Tychonoff space containing a compact resolution that swallows the compact sets \cite[Therem 2]{FK13}. Next, with the help of Theorem~\ref{P-Pn}, we greatly expand this class of lcs by showing that every uncountably-dimensional lcs with a $P$-base is also in this class if $P$ is a poset equipped with a second-countable topology such that every convergent sequence in $P$ is bounded.

We start with some definitions needed. Let $\kappa, \lambda$ be cardinals. We say that a space $X$ is $(\kappa, \lambda)_p$-equiconvergent  at a point $x\in X$ if for any indexed family $\{x_\alpha: \alpha<\kappa\}$ of sequences converging to $x$, there is a $\lambda$-sized subset $\Lambda$ such that for each neighborhood $O_x$ of $x$ there is an $n\in\omega$ such that $\{\alpha\in \Lambda: x_\alpha(n)\notin O_x\}$ is finite . We say $X$ is $(\kappa, \lambda)_p$-equiconvergent if it is $(\kappa, \lambda)_p$-equiconvergent at any point in $X$.

A lcs $E$ is said to have $(\kappa, \lambda)$-tall bornology if every $\kappa$-sized subset $A$ contains a $\lambda$-sized bounded subset. It is known (see \cite{BKP20}) that any lcs which is $(\kappa, \lambda)_p$-equicovergent has $(\kappa, \lambda)$-tall bornology. The authors in \cite{BKP20} also prove that any topological space with countable cs$^{\bullet}$-network at a point $x$ is $(\omega_1,\omega)_p$-equiconvergent at $x$.  A family $\mathcal{N}$ of subsets of $X$ is said to be a cs$^\bullet$-network at a point $x\in X$ if for any neighborhood $O_x$ of $x$ and any sequence $\{x_n: n\in \omega\}$ converging to $x$ there exists $N\in \mathcal{N}$ such that $N\subset O_x$ and $N$ contains some point $x_n$ in the sequence. Clearly any countable Pytkeev$^\ast$ network at a point $x$ is also a countable cs$^\bullet$-network at $x$. Hence any space with the strong Pytkeev$^\ast$ property has a countable cs$^\bullet$-network at each point in the space.

It is proved \cite{BKP20} that for any lcs $E$ each compact subset of $E$ has finite topological dimension if and only if each bounded linearly independent subset of $E$ is finite. We use this result to prove the following theorem.

\begin{thm} Let $P$ be a directed set equipped with a second-countable topology in which  every convergent sequence in $P$ is bounded. Every uncountably-dimensional lcs $E$ with a $P$-base contains an infinite-dimensional metrizable compact subspace.
\end{thm}

\begin{proof} Let $E$ be a lcs with a $P$-base for some directed set equipped with second-countable topology in which  every convergent sequence in $P$ is bounded. By Corollary~\ref{p_b_sPp}, $E$ has the strong Pytkeev$^\ast$ property; hence it has cs$^\bullet$-network at each point at any $x\in E$. Therefore,  $E$ is $(\omega_1, \omega)_p$-equiconvergent. By Proposition 3.3 in \cite{BKP20}, $E$ has $(\omega_1, \omega)$-tall bornology, i.e., every uncountable set in $E$ contains an infinite bounded set. Let $H$ be an uncountable Hamel basis for $E$, then $H$ contains an infinite bounded linearly independent set. Hence $E$ contains an infinite-dimensional compact subspace by Theorem 2.1 in \cite{BKP20} (the result mentioned above). By Lemma~\ref{calw}, $P$ has calibre~$(\omega_1, \omega)$. Hence any such compact subspace in $E$ is metrizable by Theorem 5.4 in \cite{Feng20}.    \end{proof}

Applying the theorem above, we obtain the following examples.

\begin{ex} Every uncountable-dimensional subspace of $C_k(\mathbb{Q})$ contains an infinite-dimensional metrizable compact subspace.  \end{ex}

Since the function space $C_k(\omega_1)$ has a $\mathcal{K}(\mathbb{Q})$-base, we get the following example.

\begin{ex} Every uncountable-dimensional subspace of $C_k(\omega_1)$ contains an infinite-dimensional metrizable compact subspace.  \end{ex}

\textbf{Acknowledgement.} The authors would like to express their gratitude to the referee for all
his/her valuable comments and suggestions which lead to the improvements of the paper.

\end{document}